\documentclass{article}
\usepackage{amsmath,amsthm}
\usepackage{amsfonts}
\usepackage{amssymb}
\usepackage{cite}
\usepackage{setspace}
\usepackage{epsfig}
\usepackage[latin1]{inputenc}
\usepackage{graphicx}
\usepackage{tikz}
\newcommand{\Sd}{\operatorname{Sd}}
\newtheorem{theorem}{Theorem}[section]

\newtheorem{corollary}[theorem]{Corollary}

\newtheorem{proposition}[theorem]{Proposition}
\newtheorem{remark}[theorem]{Remark}
\newtheorem{observation}[theorem]{Observation}

\begin{document}

\title{The Simultaneous Metric Dimension of Graph Families}
\author{Y. Ram\'{\i}rez-Cruz$^{(1)}$, O. R. Oellermann$^{(2)}$  and J. A.
Rodr\'{\i}guez-Vel\'{a}zquez$^{(1)}$\\
$^{(1)}${\small Departament d'Enginyeria Inform\`atica i Matem\`atiques,}\\
{\small Universitat Rovira i Virgili,}  {\small Av. Pa\"{\i}sos
Catalans 26, 43007 Tarragona, Spain.} \\{\small
yunior.ramirez\@@urv.cat,  juanalberto.rodriguez\@@urv.cat}\\
$^{(2)}${\small Department of Mathematics and Statistics, University of Winnipeg}\\{\small Winnipeg, MB R3B 2E9, Canada.}\\{\small o.oellermann@uwinnipeg.ca}}

\maketitle

\begin{abstract}
A vertex $v\in V$ is said to \textit{resolve} two vertices $x$ and $y$ if $d_G(v,x)\ne d_G(v,y)$.
A set $S\subset V$ is said to be a \emph{metric generator} for $G$ if any pair of vertices of $G$ is
resolved by some element of $S$. A minimum metric generator is called a \emph{metric basis}, and its cardinality, $\dim(G)$, the \emph{metric dimension} of $G$.
A set $S\subseteq V$ is said to be a simultaneous metric generator for a graph family ${\cal G}=\{G_1,G_2,\ldots,G_k\}$, defined on a common (labeled) vertex set,
if it is a metric generator for every graph of the family. A minimum cardinality simultaneous metric generator is called a simultaneous metric basis, and its
cardinality the simultaneous metric dimension of ${\cal G}$. We obtain sharp bounds for this invariants for general families of graphs and calculate closed formulae
or tight bounds for the simultaneous metric dimension of several specific graph families. For a given graph $G$ we describe a process for obtaining a lower bound
on the maximum number of graphs in a family containing $G$ that has simultaneous metric dimension equal to $\dim(G)$. It is shown that the problem of finding the simultaneous metric dimension of families of trees is $NP$-hard.
Sharp upper bounds for the simultaneous metric dimension of trees are established. The problem of finding this invariant for families of trees that can be
obtained from an initial tree by a sequence of successive edge-exchanges is considered. For such families of trees sharp upper and lower bounds for the simultaneous metric dimension are established.
\end{abstract}

\section{Introduction}
A generator of a metric space is a set $S$ of points in the space with the property that every point of the space is uniquely determined by its distances from the elements of $S$. Given a simple and connected graph $G=(V,E)$, we consider the function $d_G:V\times V\rightarrow \mathbb{N} \cup\{0\}$, where $d_G(x,y)$ is the length of a shortest path between $u$ and $v$ and $\mathbb{N}$ is the set of positive integers. Then $(V,d_G)$ is a metric space since $d_G$ satisfies (i) $d_G(x,x)=0$  for all $x\in V$, (ii) $d_G(x,y)=d_G(y,x)$  for all $x,y \in V$ and (iii) $d_G(x,y)\le d_G(x,z)+d_G(z,y)$  for all $x,y,z\in V$. A vertex $v\in V$ is said to \textit{resolve} two vertices $x$ and $y$ if $d_G(v,x)\ne d_G(v,y)$.
A set $S\subset V$ is said to be a \emph{metric generator} for $G$ if any pair of vertices of $G$ is
resolved by some element of $S$. A minimum cardinality metric generator is called a \emph{metric basis}, and its cardinality the \emph{metric dimension} of $G$, denoted by $\dim(G)$.

Motivated by the problem of uniquely determining the location of an intruder in a network, by means of a set of devices each of which can detect its distance to the intruder, the concepts of a metric generator and metric basis of a graph were introduced by Slater in \cite{Slater1975} where metric generators were called \emph{locating sets}.  Harary and Melter independently introduced the same concept in \cite{Harary1976}, where metric generators were called \emph{resolving sets}. Applications of the metric dimension to the navigation of robots in networks are discussed in \cite{Khuller1996} and applications to chemistry in \cite{Chartrand2000,Johnson1993,Johnson1998}.  This invariant was studied further in a number of other papers including, for instance \cite{Bailey2011,Caceres2007,Chartrand2000,Currie2001,Feng2012,Guo2012raey,
Haynes2006,Melter1984,Oellermann2006,Saenpholphat2004,Hernando2005,IMRAN2013,JanOmo2012,Saputro2013,Sebo2004,Yero2011}.

The navigation problem proposed in \cite{Khuller1996} deals with the movement of a robot in a ``graph space''. The
robot can locate itself by the presence of distinctively labeled ``landmarks'' in the graph
space.  On a graph,  there is neither the concept of direction nor that
of visibility. Instead, it was assumed in  \cite{Khuller1996} that a robot navigating on a graph can sense the distances
to a set of landmarks. If the robot knows its distances to a sufficiently large number of landmarks, its position
on the graph is uniquely determined. This suggests the following problem: given a graph $G$, what
are the fewest number of landmarks needed, and where should they be located, so that the
distances to the landmarks uniquely determine the robot's position on $G$?
This problem is thus equivalent to determining the metric dimension and a metric basis of $G$.

In this article we consider the following extension of this problem. Suppose that the topology of the navigation network may change within a range of possible graphs, say $G_1,G_2,...,G_k$. This scenario may reflect the use of a dynamic network whose links change over time, etc. In this case, the above mentioned problem becomes that of  determining  the minimum cardinality of a set $S$ of vertices which is simultaneously a metric generator for each graph $G_i$, $i\in  \{1,...,k\}$. So, if $S$ is a solution to this problem, then the position of a robot can be uniquely determined by the distance to the elements of $S$, regardless of the graph $G_i$ that models the network along whose edges the robot moves at each moment.

On the other hand the graphs $G_1, G_2, \ldots, G_k$ may also be the topologies of several communication networks on the same set of nodes. These communication networks may, for example, operate at different frequencies. In this case a set $S$ of nodes that resolves each $G_i$ would allow us to uniquely determine the location of an intruder into this family of networks.

Given a family ${\mathcal G}=\{G_1,G_2,...,G_k\}$ of (not necessarily edge-disjoint) connected graphs $G_i=(V,E_i)$ with common vertex set $V$ (the union of whose edge sets is not necessarily the complete graph), we define a \textit{simultaneous metric generator} for ${\mathcal{G}}$ to be a set $S\subset V$ such that $S$ is simultaneously a metric generator for each $G_i$. We say that a smallest simultaneous metric generator for ${\mathcal{G}}$ is  a \emph{simultaneous metric basis} of ${\mathcal{G}}$, and
its cardinality the \emph{simultaneous metric dimension} of ${\mathcal{G}}$, denoted by $\Sd({\mathcal{G}})$ or explicitly by $\Sd( G_1,G_2,...,G_k )$. An example is shown in Figure  \ref{ExSimultaneousBasis} where $\{v_3,v_4\}$ is a simultaneous metric basis of $\{G_1,G_2,G_3\}$.

\begin{figure}[h]
\label{Ex1SimultaneousBasis}
\begin{center}
\begin{tikzpicture}
[inner sep=0.7mm, place/.style={circle,draw=black!40,
fill=white,thick},xx/.style={circle,draw=black!99,fill=black!99,thick},
transition/.style={rectangle,draw=black!50,fill=black!20,thick},line width=1pt,scale=0.5]
\coordinate (A) at (-15,-3);
\coordinate (B) at (-13,0);
\coordinate (C) at (-8,0);
\coordinate (D) at (-15,3);

\coordinate (E) at (-5,-3);
\coordinate (F) at (-3,0);
\coordinate (G) at (2,0);
\coordinate (H) at (-5,3);

\coordinate (I) at (5,-3);
\coordinate (J) at (7,0);
\coordinate (K) at (12,0);
\coordinate (L) at (7,3);

\draw[black!40] (A)--(B) -- (D) -- (A);
\draw[black!40] (B) -- (C);

\draw[black!40] (E) -- (G) -- (H)--(E)--(F);

\draw[black!40] (I)--(J) -- (K) -- (L);

\node at (A) [place]  {};
\coordinate [label=center:{$v_1$}] (v11) at (-14,-3.1);
\node at (B) [place]  {};
\coordinate [label=center:{$v_2$}] (v12) at (-12.5,0.5);
\node at (C) [xx]  {};
\coordinate [label=center:{$v_3$}] (v13) at (-8.5,0.5);
\node at (D) [xx]  {};
\coordinate [label=center:{$v_4$}] (v14) at (-14,2.9);
\node at (E) [place]  {};
\coordinate [label=center:{$v_1$}] (v21) at (-4,-3.1);
\node at (F) [place]  {};
\coordinate [label=center:{$v_2$}] (v22) at (-2.5,0.5);
\node at (G) [xx]  {};
\coordinate [label=center:{$v_3$}] (v23) at (2.5,0.5);
\node at (H) [xx]  {};
\coordinate [label=center:{$v_4$}] (v24) at (-4,2.9);
\node at (I) [place]  {};
\coordinate [label=center:{$v_1$}] (v31) at (6,-3.1);
\node at (J) [place]  {};
\coordinate [label=center:{$v_2$}] (v32) at (7.5,0.5);
\node at (K) [xx]  {};
\coordinate [label=center:{$v_3$}] (v33) at (12.5,0.5);
\node at (L) [xx]  {};
\coordinate [label=center:{$v_4$}] (v34) at (6,2.9);

\coordinate [label=center:{$G_1$}] (G1) at (-12,-4);
\coordinate [label=center:{$G_2$}] (G2) at (-2,-4);
\coordinate [label=center:{$G_3$}] (G3) at (8,-4);

\end{tikzpicture}

\end{center}
\caption{The set $\{v_3,v_4\}$ is a simultaneous metric basis of $\{G_1,G_2,G_3\}$. Thus, $\Sd( G_1,G_2,G_3 )=2$.}
\end{figure}
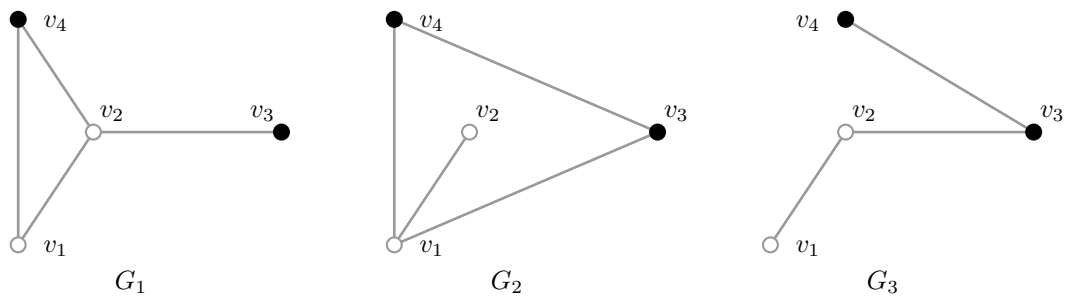

The study of simultaneous parameters in graphs was introduced by Brigham and  Dutton in \cite{Brigham1990}, where they studied simultaneous  domination. 
This should not be  confused with studies on families sharing a constant value on a parameter, for instance the study presented in \cite{IMRAN2013}, where several graph families all of whose members have the same metric dimension are studied.

We will use the notation $K_n$,  $C_n$, $N_n$ and $P_n$ to denote a complete graph,  a cycle, an empty graph, and a path  of order $n$, respectively. Let $G$ be a graph and $u,v$ vertices of $G$. We use $u \sim v$ to indicate that $u$ is adjacent with $v$ and $u \nsim v$ to indicate that $u$ is not adjacent with $v$.  The {\em diameter} of a graph $G$, denoted by $D(G)$, is the maximum distance between a pair of vertices in $G$. For the remainder of the paper, definitions will be introduced whenever a concept is needed.

\section{General Bounds}\label{sectBasicResults}

\begin{observation}\label{cotaTrivialSimultaneous}
For any  family ${\mathcal G}=\{G_1,G_2,...,G_k\}$ of  connected graphs  with common vertex set $V$ and any subfamily ${\mathcal H}$ of ${\mathcal G}$,
$$\Sd({\mathcal H})  \le \Sd({\mathcal G})\le  min \{\vert V \vert-1, \sum_{i=1}^k \dim(G_i)\}.$$
In particular, $$\max_{i\in \{1,...,k\}} \{\dim(G_i)\}\le \Sd({\mathcal G}).$$
\end{observation}

The above inequalities are sharp. For instance, for the family of graphs shown in Figure  \ref{Ex1SimultaneousBasis} we have
$\Sd(G_1,G_2,G_3)=2=\dim(G_1)=\dim(G_2)=\displaystyle\max_{i\in \{1,2,3\}}\{\dim(G_i)\}$, while for the family of graphs shown in Figure \ref{Ex2SimultaneousBasis} we have
$\Sd(G_1,G_2,G_3)=3=\vert V\vert -1.$

The following result is a direct consequence of Observation~\ref{cotaTrivialSimultaneous}.

\begin{corollary}\label{corollaryCompleteInFamily}
Let ${\mathcal G}$ be a family of  connected graphs  with  the same vertex set. If $K_n\in {\mathcal G}$, then
$$\Sd({\mathcal G})=n-1.$$
\end{corollary}


As shown in Figure \ref{Ex2SimultaneousBasis}, the converse of Corollary \ref{corollaryCompleteInFamily} does not hold.

\begin{figure}[h]\label{Ex2SimultaneousBasis}
\begin{center}
\begin{tikzpicture}
[inner sep=0.7mm, place/.style={circle,draw=black!40,
fill=white,thick},xx/.style={circle,draw=black!99,fill=black!99,thick},
transition/.style={rectangle,draw=black!50,fill=black!20,thick},line width=1pt,scale=0.5]
\coordinate (A) at (-15,-3);
\coordinate (B) at (-9,-3);
\coordinate (C) at (-9,3);
\coordinate (D) at (-15,3);

\coordinate (E) at (-5,-3);
\coordinate (F) at (1,-3);
\coordinate (G) at (1,3);
\coordinate (H) at (-5,3);

\coordinate (I) at (5,-3);
\coordinate (J) at (11,-3);
\coordinate (K) at (11,3);
\coordinate (L) at (5,3);

\draw[black!40] (A)--(B) -- (D) -- (A);
\draw[black!40] (B) -- (C)--(D);

\draw[black!40] (F)--(E);
\draw[black!40] (H)--(F)--(G);

\draw[black!40] (K)--(L)--(I);
\draw[black!40] (L)--(J);

\node at (A) [place]  {};
\coordinate [label=center:{$v_1$}] (v11) at (-15.5,-2.5);
\node at (B) [xx]  {};
\coordinate [label=center:{$v_2$}] (v12) at (-8.3,-2.5);
\node at (C) [xx]  {};
\coordinate [label=center:{$v_3$}] (v13) at (-8.3,2.5);
\node at (D) [xx]  {};
\coordinate [label=center:{$v_4$}] (v14) at (-15.5,2.5);
\node at (E) [place] {};
\coordinate [label=center:{$v_1$}] (v21) at (-5.5,-2.5);
\node at (F) [xx]  {};
\coordinate [label=center:{$v_2$}] (v22) at (1.7,-2.5);
\node at (G) [xx]  {};
\coordinate [label=center:{$v_3$}] (v23) at (1.7,2.5);
\node at (H) [xx]  {};
\coordinate [label=center:{$v_4$}] (v24) at (-5.5,2.5);
\node at (I) [place] {};
\coordinate [label=center:{$v_1$}] (v31) at (4.5,-2.5);
\node at (J) [xx]  {};
\coordinate [label=center:{$v_2$}] (v32) at (11.7,-2.5);
\node at (K) [xx]  {};
\coordinate [label=center:{$v_3$}] (v33) at (11.7,2.5);
\node at (L) [xx]  {};
\coordinate [label=center:{$v_4$}] (v34) at (4.5,2.5);

\coordinate [label=center:{$G_1$}] (G1) at (-12,-4);
\coordinate [label=center:{$G_2$}] (G2) at (-2,-4);
\coordinate [label=center:{$G_3$}] (G3) at (8,-4);

\end{tikzpicture}

\end{center}
\caption{The set $\{v_2,v_3,v_4\}$ is a simultaneous metric basis of $\{G_1,G_2,G_3\}$. Thus, $\Sd(G_1,G_2,G_3)=3=n-1$.}
\end{figure}
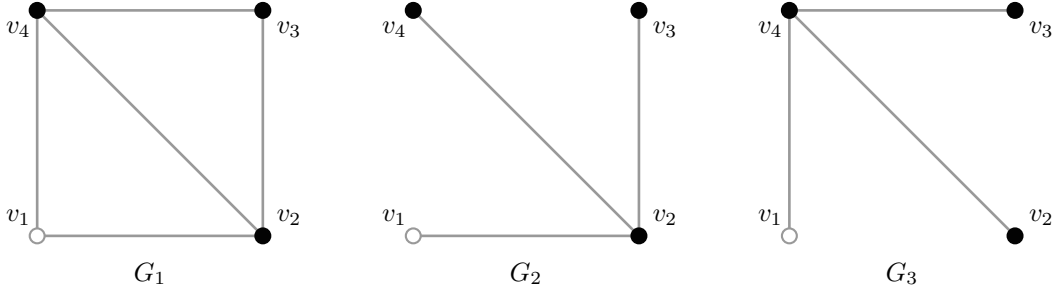



Given a graph $G=(V,E)$ and a vertex $v\in V,$ the set $N_G(v)=\{u\in V:\; u\sim v\}$ is the {\em open neighborhood} of $v$ and the set $N_G[v] = N_G(v)\cup \{v\}$ is the {\em closed neighborhood} of $v$.
Two vertices $x,y\in V(G)$  are \textit{twin vertices} in $G$  if $N_G(x)=N_G(y)$ or $N_G[x]=N_G[y]$.
\begin{theorem}\label{ExtremeCaseAbove}
Let ${\mathcal G}$ be a  family of connected graphs  with the same vertex set $V$. Then $\Sd({\mathcal G})=\vert V\vert-1$ if and only if for every pair $u,v\in V$, there exists a graph $G_{uv}\in {\mathcal G}$ such that $u$ and $v$ are twins in $G_{uv}$.
\end{theorem}
\begin{proof}
We first note that for any connected graph $G=(V,E)$ and any vertex $v\in V$ the set $V-\{v\}$ is a metric generator for $G$. So, if $\Sd({\mathcal G})=\vert V\vert-1$, then for every $v\in V$, the set  $V-\{v\}$ is a simultaneous metric basis of ${\mathcal G}$ and, as a consequence, for every $u\in V-\{v\}$ there exists a graph $G_{uv}\in {\mathcal G}$ such that the set $V-\{u,v\}$ is not a metric generator for $G_{uv}$, \textit{i.e.},
for every $x\in V-\{u,v\}$ we have $d_{G_{u,v}}(u,x)=d_{G_{u,v}}(v,x)$. So $u$ and $v$ must be twin vertices in $G_{u,v}$.

Conversely, if for every $u,v\in V$ there exists a graph $G_{uv}\in {\mathcal G}$ such that $u$ and $v$ are twin vertices in $G_{uv}$, then for any simultaneous metric basis $B$ of $ {\mathcal G}$ either $u\in B$ or $v\in B$. Hence, all but one element of $V$ must belong to $B$. Therefore $\vert B \vert\ge \vert V\vert -1$ and, by Observation \ref{cotaTrivialSimultaneous}, we conclude that $\Sd({\mathcal G})=\vert V\vert-1$.
\end{proof}

Notice that Corollary \ref{corollaryCompleteInFamily} is also a consequence of Theorem \ref{ExtremeCaseAbove} as is the next result.

\begin{corollary}
 Let ${\mathcal G}$ be a  family of connected graphs  with the same vertex set $V$. If ${\mathcal G}$ contains three star graphs having different centers, then
$\Sd({\mathcal G})=\vert V\vert-1$.
\end{corollary}


It was shown in \cite{Chartrand2000} that for any connected graph $G$ of order $n$ and diameter $D(G)$,
\begin{equation}\label{cotadiametro}
\dim (G)\le n-D(G).
\end{equation}
Our next result is an extension of (\ref{cotadiametro}) to the case of the simultaneous metric dimension.
\begin{theorem}
 Let ${\mathcal G}$ be a  family of graphs with common vertex set $V$ that have a shortest path of length $d$ in common.    Then  $$ \Sd({\mathcal G})\le \vert V\vert-d.$$
\end{theorem}

\begin{proof}
Let ${\mathcal G}=\{G_1,G_2,...,G_k\}$ be a  family of graphs with common vertex set $V$ having a  shortest path $v_0,v_1,...,v_d$  in common. Let $W=V-\{v_1,...,v_d\}$. Since $d_{G_j}(v_0,v_i)=i$,  for $i\in \{1,...,d\}$, we conclude that $W$ is  a metric generator for each $G_j$.
Hence, $ \Sd({\mathcal G})\le \vert W\vert= \vert V\vert-d.$
\end{proof}

Let $r \ge 3$ be an integer. Label the vertices of $K_r$ and $K_{1,r-1}$ with the same set of labels and suppose $c$ is the label of the centre of the star $K_{1,r-1}$. Let $P_d$, $d \ge 2$, be an $a$--$b$ path of order $d$ whose vertex set is disjoint from that of $K_r$. Let $G_1$ be the graph obtained from the complete graph $K_r=(V',E')$, $r\ge 3$, and the path graph $P_d$, $d\ge 2$,  by identifying the leaf $a$ of $P_d$,  with the vertex $c$ of $K_r$ and calling it $c$, and let $G_2$ be the graph obtained by  identifying the leaf $a$ of $P_d$ with the center $c$ of the star $K_{1,r-1}$ and also calling it $c$. In this case, $G_1$ and  $G_2$ have the same vertex set $V$ (where $\vert V \vert=d+r-1$). For any $v\in V(K_r)-\{c\}$ we have $d_{G_1}(b,v)=d_{G_2}(b,v)=d$ and  $ V(P_d)\cup \{v\}$ is a shortest path of length $d$ in both graphs  $G_1$ and $G_2$. Moreover, $W=(V'-\{v,c\})\cup \{b\}$  is a simultaneous metric basis of $\{G_1,G_2\}$ and so $\Sd(G_1,G_2)=\vert V \vert-d$. Therefore, the above bound is sharp.

\section{Simultaneous Metric Dimension of Families of Graphs with Small Metric Dimension}
In this section we focus on families of graphs on the same vertex set each of which have dimension 1 or 2. It was shown in \cite{Chartrand2000} that  $\dim(G)=1$ if and only if $G$ is a path. The first result in this section deals with families of graphs for which the simultaneous metric dimension is as small as possible.

\begin{theorem}\label{familyPaths}
Let ${\mathcal G}$ be a  family of connected graphs  on a common vertex set. Then
\begin{enumerate}
\item $\Sd({\mathcal G})=1$ if and only if ${\mathcal G}$ is a collection of paths that share a common leaf.
\item If  ${\mathcal G}$ is a collection of paths, then $1\le \Sd({\mathcal G})\le 2.$
\end{enumerate}
\end{theorem}
\begin{proof}
 Thus, if $\Sd({\mathcal G})=1$, then the family ${\mathcal G}$ is a collection of paths.

Moreover, if $v$ is a vertex of degree $2$ in a path $P$, then $v$ does not distinguish its neighbours and, as a consequence, $\{v\}$ is a metric basis of $P$ if and only if $v$ is a leaf of $P$.  Therefore, 1. follows.

Since any path has metric dimension $1$, and any pair of distinct vertices of a path $P$ is a metric generator for $P$, we conclude that  2. follows.
\end{proof}

\begin{theorem}\label{pathsAndMore}
 Let ${\mathcal G}$ be a  family of graphs on a common vertex set $V$ such that $\mathcal{G}$ does not only consist of paths.  Let ${\mathcal H}$ be the collection of elements of  ${\mathcal G}$ which are not paths.   Then  $$ \Sd({\mathcal G})=\Sd({\mathcal H}).$$
\end{theorem}
\begin{proof}
Since ${\mathcal H}$ is a non-empty subfamily of ${\mathcal G}$ we conclude that $\Sd({\mathcal G})\ge \Sd({\mathcal H}).$
From Theorem \ref{familyPaths}(1), it follows that $\Sd({\mathcal H})\ge 2$. Moreover, as any pair of vertices of a path $P$  is a metric generator for $P$, it follows that if $B\subseteq V$ is a simultaneous metric basis of ${\mathcal H}$, then $B$ is a simultaneous metric generator for  ${\mathcal G}$ and, as a result,
$\Sd({\mathcal G})\le \vert B\vert =\Sd({\mathcal H}).$
\end{proof}

\begin{theorem}\label{familyCycles}
 Let ${\mathcal G}=\{G_1,G_2,...,G_k\}$ be a  family of cycles on a common vertex set $V.$ Then the following assertions hold:
 \begin{enumerate}
\item If  $\vert V\vert$ is odd, then $ \Sd({\mathcal G})=2$.
\item If  $\vert V\vert$ is even, then  $2 \le \Sd({\mathcal G})\le 3$. Moreover,  for $\vert V\vert$  even, $\Sd({\mathcal G})=2$ if and only if there are two vertices $u,v\in V$ which are not mutually antipodal in $G_i$ for every $i\in \{1,...,k\}$.
\item If $\vert V\vert$ is even and $k < n-1$, then $\Sd({\mathcal{G}}) = 2$. Moreover, this result is best possible in the sense that there is a family of $(n-1)$ cycles of order $n$ on the same vertex set whose simultaneous metric dimension is $3$.
\end{enumerate}
\end{theorem}
\begin{proof}
The result is clear for $\vert V\vert=3$. Let $C_n$ be a cycle of order $\vert V\vert=n\ge 4$. We first assume  that $n$ is odd. In this case, given   four different vertices $u,v,x,y\in V(C_n)$ we have $d_{C_n}(u,x)\ne d_{C_n}(u,y)$ or $d_{C_n}(v,x)\ne d_{C_n}(v,y)$. Hence, we conclude that $\{u,v\}$ is a metric generator for $C_n$ and, since $\dim(C_n)>1$, we conclude that $\{u,v\}$ is a metric basis for $C_n$. Thus, $\{u,v\}$ is a simultaneous metric basis for ${\mathcal G}$.
Therefore, in this case $\Sd({\mathcal G})= 2$. Thus 1. holds.

From now on we assume that  $\vert V\vert =n$ is even. Note that in this case every $G_i$ is a $2$-antipodal\footnote{We recall that $G=(V(G),E(G))$ is $2$-antipodal if for each vertex $x\in V(G)$ there exists exactly one vertex $y\in V(G)$ such that $d_G(x,y)=D(G)$.}  graph. Let $u,v\in V(C_n)$ be two vertices which are not mutually antipodal in $C_n$. Since, for every pair of distinct vertices $x,y\in V(C_n)$, we have $d_{C_n}(u,x)\ne d_{C_n}(u,y)$ or $d_{C_n}(v,x)\ne d_{C_n}(v,y)$, we conclude that $\{u,v\}$ is a metric generator for $C_n$ and, since $\dim(C_n)>1$, we conclude that $\{u,v\}$ is a metric basis. Clearly, no pair of mutually antipodal vertices form a metric basis for $C_n$. Therefore, $ \Sd({\mathcal G})=2$ if and only if there are two vertices $u,v\in V$ which are not mutually antipodal in $G_i$  for every $i\in \{1,...,k\}$. Suppose that, for every pair of distinct vertices $u,v\in V$, there exists $G_i\in {\mathcal G}$ such that $u$ and $v$ are mutually antipodal in $G_i$. In this case we have $\Sd({\mathcal G})\ge 3$. Now, since for  three different vertices $u,v,w\in V$, only two of them may be  mutually antipodal in $G_i$, we conclude that $\{u,v,w\}$ is a simultaneous metric generator for ${\mathcal G}$. Therefore, in this case, $\Sd({\mathcal G})= 3$. This completes the proof of 2.

Since each of the $k$ cycles in $\mathcal{G}$ has $n/2$ antipodal pairs  it follows that if $k < n-1$ or equivalently $\frac{nk}{2} < {n \choose 2}$, then $\Sd({\mathcal{G}}) = 2$. This inequality is best possible in the sense that there is a collection of $(n-1)$ cycles ${\mathcal{G}}=\{C_1',C_2', \ldots, C_{n-1}'\}$ with vertex set $\{1,2, \ldots, n\}$ such that each of the ${n \choose 2}$ possible pairs from $\{1,2, \ldots, n\}$ is an antipodal pair on exactly one of these cycles and hence $\Sd({\mathcal{G}})=3$. We construct the labeling of these cycles by assigning pairs of labels to antipodal pairs in such a way that a given pair is assigned to exactly one of these $(n-1)$ cycles. Consider the upper triangular array whose $(i,j)^{th}$ entry is $(i,j)$ for $1 \le i < j \le n$. Select the first non-empty entry in row $1$. This entry is the ordered pair $(1,2)$. Begin by assigning the labels $1$ and $2$ to the vertices in positions 1 and $n/2$ on $C_1'$. Now mark rows and columns $1$ and $2$ used and mark the pair $(1,2)$ as unavailable. Find the first unused row and subject to this the first unused column and let the corresponding entry in the array be say  $(i_{1_2},j_{1_2})$. Assign $i_{1_2}$ and $j_{1_2}$ to vertices in positions $2$ and $1+n/2$ on $C_1'$ and mark both rows and columns $i_{1_2}$ and $j_{1_2}$ as used and the pair $(i_{1_2},j_{1_2})$ as unavailable. Next find the first available pair in the first unused row and subject to this in an unused column, say $(i_{1_3},j_{1_3})$. Assign the labels $i_{1_3}$ and $j_{1_3}$ to the vertices in $C_1'$ in positions $3$ and $2+n/2$, respectively. We continue this process until all rows and columns of the array have been marked used. Moreover, whenever the entries of an ordered pair are used as labels of vertices in $C_1'$ we mark that pair as unavailable. Now reset the labels on all rows and columns in the array as unused but do not reset the labels on the ordered pairs. Next find the first  available entry  say $(i_{2_1},j_{2_1})$ in row 1 and assign $i_{2_1}$ and $j_{2_1}$ to the vertices in positions $1$ and $n/2$, respectively, of $C_2'$. Mark  rows and columns $i_{2_1}$ and $j_{2_1}$ as used and mark the pair $(i_{2_1},j_{2_1})$ as unavailable. Now find the first non-empty available entry  in the first unmarked row and subject to this in the first unmarked column, say $(i_{2_2},j_{2_2})$, and assign $i_{2_2}$ and $j_{2_2}$ to vertices in positions $2$ and $1+n/2$ in $C_2'$. Continue in this manner until entries of each ordered pair in the triangular array have been assigned as labels to antipodal vertices in one of the cycles in $\mathcal{G}$. Then $\Sd({\mathcal{G}}) = 3$. This completes the proof of 3.

\end{proof}

As a direct consequence of Theorems~\ref{pathsAndMore} and~\ref{familyCycles}  we can obtain the following result for the special case of families composed solely of paths and cycles.

\begin{corollary}
Let ${\mathcal G}=\{G_1,G_2,...,G_k\}$ be a  family of  cycles and paths  with the same vertex set that contains at least one cycle. Then the following assertions hold:
 \begin{enumerate}
\item If  $\vert V\vert$ is odd, then $ \Sd({\mathcal G})=2$.
\item If  $\vert V\vert$ is even, then  $2 \le \Sd({\mathcal G})\le 3.$ Moreover,  for $\vert V \vert$ even,
$\Sd({\mathcal G})=2$ if and only if there are two vertices $u,v \in V$ which are not mutually antipodal in $G_i$ for every cycle $G_i\in {\mathcal G}$.
\item If  $\vert V\vert$ is even and $\mathcal{G}$ contains fewer than $(n-1)$ cycles, then $\Sd({\mathcal{G}}) =2$. Moreover, there is a family $\mathcal{G}$ containing $(n-1)$ cycles such that $\Sd({\mathcal{G}})=3$.
\end{enumerate}
\end{corollary}

\section{Large Families of Graphs with a Fixed Basis and Large Common Induced Subgraph}

In this section we show that there may be large families of graphs on the same vertex set with small simultaneous metric dimension.
We accomplish this by describing a general approach for constructing large families of labeled graphs on the same vertex set for which the simultaneous metric dimension attains the lower bound given in Observation \ref{cotaTrivialSimultaneous}. Moreover we show that such a family of graphs contain large isomorphic common induced subgraphs.

Let $G=(V,E)$ be a graph and let  $Perm(V)$ be the set of all permutations of $V$.
 Given a subset $X\subseteq V$, the stabilizer of $X$ is the set of permutations
 ${\cal S}(X)=\{f\in Perm(V): f(x)=x, \; \mbox{\rm for every } x\in X\}$. As usual, we denote by $f(X)$ the image of a subset $X$ under $f$, \textit{i.e}., $f(X)=\{f(x):\; x\in X\}$.

Let $B$ be metric basis of a graph $G=(V,E)$ of diameter $D(G)$. For any $r \in \{0,1,...,D(G)\}$ we define the set $${\rm \bf B}_{r}(B)=\displaystyle\bigcup_{x\in B}\{y\in V:\; d_G(x,y)\le r\}.$$
In particular, ${\rm {\bf  B}}_{0}(B)=B$ and ${\rm {\bf  B}}_{1}(B)=\displaystyle\bigcup_{x\in B}N_G[x] $. Moreover,  since $B$ is a metric basis of $G$, $|{\rm {\bf  B}}_{D(G)-1}(B)|\ge |V|-1$.

Let $G$ be a connected graph that is not complete. Given a permutation $f\in {\cal{S}}(B)$ of $V$ we say that a graph $G'=(V,E')$ belongs to the  family  ${\cal {G}}_f$ if and only if $N_{G'}(f(v))=f(N_G(v))$, for every $v\in {\rm {\bf  B}}_{D(G)-2}(B)$. In particular, if $D(G)=2$ and $f\in {\cal S}(B)$, then  $G'=(V,E')$ belongs to the  family  ${\cal  G}_f$ if and only if $N_{G'}(x)=f(N_G(x))$, for every $x\in B$. Moreover, if $G$ is a complete graph, we defined ${\cal{G}}_f=\{G\}$.

\begin{remark}\label{path_preserving}
Let $B$ be a metric basis of a connected non-complete graph $G$,  let $f\in {\cal{S}}(B)$ and  $G'\in  {\cal {G}}_f$. Then for any $b\in B$ and $k\in \{1,...,D(G)-1\}$, a sequence $b=v_0,v_1,...,v_{k-1},v_k=v$ is a path in $G$ if and only if the sequence $b=f(v_0),f(v_1),...,f(v_{k-1}),f(v_k)=f(v)$ is a path in $G'$.
\end{remark}
\begin{proof}
Let  $b\in {\rm \bf  B}$.  Since $G'\in {\cal  G}_f$ and $b=v_0 \in {\rm {\bf  B}}_{D(G)-2}(B)$, we have that  $f(v_{1})\in N_{G'}(f(v_{0}))$ if and only if $v_{1}\in N_G(v_0)$ and, in general, if $v_i \in {\rm {\bf  B}}_{D(G)-2}(B)$, then $f(v_{i+1})\in N_{G'}(f(v_{i}))$ if and only if $v_{i+1}\in N_G(v_i)$. Therefore, for any $k\in \{1,...,D(G)-1\}$, a sequence $(b=)f(v_0),f(v_1),...,f(v_{k-1}),$ $f(v_k)(=f(v))$ is a path in $G'$  if and only if $(b=)v_0,v_1,...,v_{k-1},v_k(=v)$ is a path in $G$.
\end{proof}

 \begin{corollary}\label{Lemma-Paths-in-G'}
 Let $B$ be a metric basis of a connected  graph $G$,  let $f\in {\cal S}(B)$ and  $G'\in  {\cal  G}_f$. Then for any $b\in B$ and $v\in {\rm {\bf  B}}_{D(G)-1}(B)$,
 $d_G(b,v)=k$ if and only if $d_{G'}(b,f(v))=k.$
 \end{corollary}

 \begin{corollary}\label{Lemma_induced_subgraphs_inG'}
 Let $B$ be a metric basis of a connected  graph $G$,  let $f\in {\cal{S}}(B)$ and  $G'\in  {\cal{G}}_f$. Then $\langle {\rm {\bf{B}}}_{D(G)-2}(B) \rangle \cong \langle {\rm {\bf{B}}}_{D(G')-2}(B) \rangle$.
 \end{corollary}
 \begin{proof}
 Since $G'\in  {\cal  G}_f$, the function $f$ is a bijection from $V(G)$ onto $V(G')$. It remains to show that the restriction of $f$ to $\langle {\rm {\bf{B}}}_{D(G)-2}(B) \rangle$ is an isomorphism, i.e., we need to show that $uv$ is an edge of $\langle {\rm {\bf{B}}}_{D(G)-2}(B) \rangle$ if and only if $f(u)f(v)$ is an edge of $\langle {\rm {\bf{B}}}_{D(G')-2}(B) \rangle$. Let $u, v \in {\rm {\bf{B}}}_{D(G)-2}(B)$. Let $k$ be the length of a shortest path from the set $\{u,v\}$ to the set $B$. Then there is a $b \in B$ such that $k = min\{d_G(b,u), d_G(b,v)\} \le D(G)-2$. We may assume $d_G(b,u) = k$.
 So there is a path $(b=)v_0,v_1,...,v_{k-1},v_k(=u)$ in $\langle {\rm {\bf  B}}_{D(G)-2}(B)\rangle$. By Remark \ref{path_preserving} $(b=)v_0,v_1,...,v_{k-1},v_k(=u), v$ is a path in $G$ if and only if $(b=)f(v_0),f(v_1),...,f(v_{k-1}),f(v_k)(=f(u)), f(v)$ is a path in $G'$. So $uv \in E(\langle {\rm {\bf  B}}_{D(G)-2}(B) \rangle)$ if and only if $f(u)f(v) \in E(\langle {\rm {\bf  B}}_{D(G')-2}(B) \rangle)$.
 \end{proof}

 Now we define a family of graphs ${\cal{G}}_B$, associated with $B$, as follows.
$${\cal{G}}_B=\displaystyle\bigcup_{f\in {\cal S}(B)}{\cal{G}}_f.$$

Notice that if  ${\rm {\bf{B}}}_{D(G)-2}(B)\subsetneq V$, then  any graph $G'\in {\cal{G}}_B$ is isomorphic to a graph $G^*=(V,E^*)$ whose edge set $E^*$ can be partitioned into two sets $E^*_1$, $E^*_2$, where $E^*_1$ consists of all edges of $G$ having at least one vertex in ${\rm {\bf  B}}_{D(G)-2}(B)$ and $E^*_2$ is a subset of edges of a complete graph whose vertex set is $V-{\rm {\bf{B}}}_{D(G)-2}(B)$. Hence, ${\cal{G}}_B$ contains
$2^{\frac{l(l-1)}{2}}|V-B|!$ different labeled graphs, where $l=|V-{\rm {\bf  B}}_{D(G)-2}(B)|$.
Clearly, if $|{\rm {\bf  B}}_{D(G)-1}(B)|=|V|$, then all these graphs are connected and if $|{\rm {\bf  B}}_{D(G)-1}(B)|=|V|-1$, then
 $2^{\frac{(l-1)(l-2)}{2}}(2^{l-1}-1)|V-B|!$ of these graphs are connected.

\begin{figure}[h]
\label{ExSimultaneousBasis}
\begin{center}
\begin{tikzpicture}
[inner sep=0.7mm, place/.style={circle,draw=black!40,
fill=white,thick},xx/.style={circle,draw=black!99,fill=black!99,thick},
transition/.style={rectangle,draw=black!50,fill=black!20,thick},line width=1pt,scale=0.5]


\coordinate (A) at (1,20);
\coordinate (B) at (4.5,17);
\coordinate (C) at (4.5,18.5);
\coordinate (D) at (4.5,20);
\coordinate (E) at (4.5,21.5);
\coordinate (F) at (4.5,23);

\draw[black!40] (A) -- (B) -- (C) -- (D) -- (E) -- (F) -- (A);
\draw[black!40] (A) -- (C);
\draw[black!40] (A) -- (D);
\draw[black!40] (A) -- (E);

\node at (A) [place] {};
\node at (B) [xx] {};
\node at (C) [place] {};
\node at (D) [place] {};
\node at (E) [place] {};
\node at (F) [xx] {};

\coordinate [label=left:{$v_6$}] (v6) at (0.5,20);
\coordinate [label=left:{$v_1$}] (v1) at (6,17);
\coordinate [label=left:{$v_2$}] (v2) at (6,18.5);
\coordinate [label=left:{$v_3$}] (v3) at (6,20);
\coordinate [label=left:{$v_4$}] (v4) at (6,21.5);
\coordinate [label=left:{$v_5$}] (v5) at (6,23);

\coordinate [label=left:{$G$}] (G) at (3.5,16.5);


\coordinate [label=center:{$f$}] (f) at (11.5,23);
\coordinate [label=center:{$v_1 \rightarrow v_1$}] (f11) at (11.5,22);
\coordinate [label=center:{$v_2 \rightarrow v_4$}] (f24) at (11.5,21);
\coordinate [label=center:{$v_3 \rightarrow v_2$}] (f32) at (11.5,20);
\coordinate [label=center:{$v_4 \rightarrow v_6$}] (f46) at (11.5,19);
\coordinate [label=center:{$v_5 \rightarrow v_5$}] (f55) at (11.5,18);
\coordinate [label=center:{$v_6 \rightarrow v_3$}] (f63) at (11.5,17);



[inner sep=0.7mm, place/.style={circle,draw=black!40,
fill=white,thick},xx/.style={circle,draw=black!99,fill=black!99,thick},
transition/.style={rectangle,draw=black!50,fill=black!20,thick},line width=1pt,scale=0.5]


\coordinate (A1) at (1,4);
\coordinate (B1) at (4.5,1);
\coordinate (C1) at (4.5,2.5);
\coordinate (D1) at (4.5,4);
\coordinate (E1) at (4.5,5.5);
\coordinate (F1) at (4.5,7);

\draw[black!40] (A1)--(B1) -- (C1) -- (D1) -- (E1) -- (F1) -- (A1);
\draw[black!40] (A1) -- (C1);
\draw[black!40] (A1) -- (D1);
\draw[black!40] (A1) -- (E1);

\node at (A1) [place] {};
\node at (B1) [xx] {};
\node at (C1) [place] {};
\node at (D1) [place] {};
\node at (E1) [place] {};
\node at (F1) [xx] {};

\coordinate [label=left:{$v_3$}] (v31) at (0.5,4);
\coordinate [label=left:{$v_1$}] (v11) at (6,1);
\coordinate [label=left:{$v_4$}] (v41) at (6,2.5);
\coordinate [label=left:{$v_2$}] (v21) at (6,4);
\coordinate [label=left:{$v_6$}] (v61) at (6,5.5);
\coordinate [label=left:{$v_5$}] (v51) at (6,7);

\coordinate [label=left:{$G_5$}] (G1) at (3.5,0.5);


\coordinate (A2) at (8,4);
\coordinate (B2) at (11.5,1);
\coordinate (C2) at (11.5,2.5);
\coordinate (D2) at (11.5,4);
\coordinate (E2) at (11.5,5.5);
\coordinate (F2) at (11.5,7);

\draw[black!40] (A2)--(B2) -- (C2) -- (D2) -- (E2) -- (F2) -- (A2);
\draw[black!40] (A2) -- (C2);
\draw[black!40] (A2) -- (E2);

\node at (A2) [place] {};
\node at (B2) [xx] {};
\node at (C2) [place] {};
\node at (D2) [place] {};
\node at (E2) [place] {};
\node at (F2) [xx] {};

\coordinate [label=left:{$v_3$}] (v32) at (7.5,4);
\coordinate [label=left:{$v_1$}] (v12) at (13,1);
\coordinate [label=left:{$v_4$}] (v42) at (13,2.5);
\coordinate [label=left:{$v_2$}] (v22) at (13,4);
\coordinate [label=left:{$v_6$}] (v62) at (13,5.5);
\coordinate [label=left:{$v_5$}] (v52) at (13,7);

\coordinate [label=left:{$G_6$}] (G2) at (10.5,0.5);


\coordinate (A3) at (15,4);
\coordinate (B3) at (18.5,1);
\coordinate (C3) at (18.5,2.5);
\coordinate (D3) at (18.5,4);
\coordinate (E3) at (18.5,5.5);
\coordinate (F3) at (18.5,7);

\draw[black!40] (A3)--(B3) -- (C3) -- (A3) -- (E3) -- (F3) -- (A3);
\draw[black!40] (A3) -- (D3);

\node at (A3) [place] {};
\node at (B3) [xx] {};
\node at (C3) [place] {};
\node at (D3) [place] {};
\node at (E3) [place] {};
\node at (F3) [xx] {};

\coordinate [label=left:{$v_3$}] (v33) at (14.5,4);
\coordinate [label=left:{$v_1$}] (v13) at (20,1);
\coordinate [label=left:{$v_4$}] (v43) at (20,2.5);
\coordinate [label=left:{$v_2$}] (v23) at (20,4);
\coordinate [label=left:{$v_6$}] (v63) at (20,5.5);
\coordinate [label=left:{$v_5$}] (v53) at (20,7);

\coordinate [label=left:{$G_7$}] (G3) at (17.5,0.5);


\coordinate (A4) at (22,4);
\coordinate (B4) at (25.5,1);
\coordinate (C4) at (25.5,2.5);
\coordinate (D4) at (25.5,4);
\coordinate (E4) at (25.5,5.5);
\coordinate (F4) at (25.5,7);

\draw[black!40] (A4)--(B4) -- (C4) -- (D4) -- (A4) -- (E4) -- (F4) -- (A4);

\node at (A4) [place] {};
\node at (B4) [xx] {};
\node at (C4) [place] {};
\node at (D4) [place] {};
\node at (E4) [place] {};
\node at (F4) [xx] {};

\coordinate [label=left:{$v_3$}] (v34) at (21.5,4);
\coordinate [label=left:{$v_1$}] (v14) at (27,1);
\coordinate [label=left:{$v_4$}] (v44) at (27,2.5);
\coordinate [label=left:{$v_2$}] (v24) at (27,4);
\coordinate [label=left:{$v_6$}] (v64) at (27,5.5);
\coordinate [label=left:{$v_5$}] (v54) at (27,7);

\coordinate [label=left:{$G_8$}] (G4) at (24.5,0.5);


\coordinate (A5) at (1,12);
\coordinate (B5) at (4.5,9);
\coordinate (C5) at (4.5,10.5);
\coordinate (D5) at (4.5,12);
\coordinate (E5) at (4.5,13.5);
\coordinate (F5) at (4.5,15);

\draw[black!40] (A5) -- (B5) -- (C5) -- (D5) -- (E5) -- (F5) -- (A5);
\draw[black!40] (E5) arc (90:270:1.5);

\node at (A5) [place] {};
\node at (B5) [xx] {};
\node at (C5) [place] {};
\node at (D5) [place] {};
\node at (E5) [place] {};
\node at (F5) [xx] {};

\coordinate [label=left:{$v_3$}] (v35) at (0.5,12);
\coordinate [label=left:{$v_1$}] (v15) at (6,9);
\coordinate [label=left:{$v_4$}] (v45) at (6,10.5);
\coordinate [label=left:{$v_2$}] (v25) at (6,12);
\coordinate [label=left:{$v_6$}] (v65) at (6,13.5);
\coordinate [label=left:{$v_5$}] (v55) at (6,15);

\coordinate [label=left:{$G_1$}] (G5) at (3.5,8.5);


\coordinate (A6) at (8,12);
\coordinate (B6) at (11.5,9);
\coordinate (C6) at (11.5,10.5);
\coordinate (D6) at (11.5,12);
\coordinate (E6) at (11.5,13.5);
\coordinate (F6) at (11.5,15);

\draw[black!40] (A6) -- (B6) -- (C6) -- (D6) -- (E6) -- (F6) -- (A6);
\draw[black!40] (A6) -- (E6);

\node at (A6) [place] {};
\node at (B6) [xx] {};
\node at (C6) [place] {};
\node at (D6) [place] {};
\node at (E6) [place] {};
\node at (F6) [xx] {};

\coordinate [label=left:{$v_3$}] (v36) at (7.5,12);
\coordinate [label=left:{$v_1$}] (v16) at (13,9);
\coordinate [label=left:{$v_4$}] (v46) at (13,10.5);
\coordinate [label=left:{$v_2$}] (v26) at (13,12);
\coordinate [label=left:{$v_6$}] (v66) at (13,13.5);
\coordinate [label=left:{$v_5$}] (v56) at (13,15);

\coordinate [label=left:{$G_2$}] (G6) at (10.5,8.5);


\coordinate (A7) at (15,12);
\coordinate (B7) at (18.5,9);
\coordinate (C7) at (18.5,10.5);
\coordinate (D7) at (18.5,12);
\coordinate (E7) at (18.5,13.5);
\coordinate (F7) at (18.5,15);

\draw[black!40] (E7) -- (F7) -- (A7) -- (B7) -- (C7) -- (D7) -- (A7);

\node at (A7) [place] {};
\node at (B7) [xx] {};
\node at (C7) [place] {};
\node at (D7) [place] {};
\node at (E7) [place] {};
\node at (F7) [xx] {};

\coordinate [label=left:{$v_3$}] (v37) at (14.5,12);
\coordinate [label=left:{$v_1$}] (v17) at (20,9);
\coordinate [label=left:{$v_4$}] (v47) at (20,10.5);
\coordinate [label=left:{$v_2$}] (v27) at (20,12);
\coordinate [label=left:{$v_6$}] (v67) at (20,13.5);
\coordinate [label=left:{$v_5$}] (v57) at (20,15);

\coordinate [label=left:{$G_3$}] (G7) at (17.5,8.5);


\coordinate (A8) at (22,12);
\coordinate (B8) at (25.5,9);
\coordinate (C8) at (25.5,10.5);
\coordinate (D8) at (25.5,12);
\coordinate (E8) at (25.5,13.5);
\coordinate (F8) at (25.5,15);

\draw[black!40] (E8) -- (F8) -- (A8) -- (B8) -- (C8) -- (D8);
\draw[black!40] (E8) arc (90:270:1.5);

\node at (A8) [place] {};
\node at (B8) [xx] {};
\node at (C8) [place] {};
\node at (D8) [place] {};
\node at (E8) [place] {};
\node at (F8) [xx] {};

\coordinate [label=left:{$v_3$}] (v38) at (21.5,12);
\coordinate [label=left:{$v_1$}] (v18) at (27,9);
\coordinate [label=left:{$v_4$}] (v48) at (27,10.5);
\coordinate [label=left:{$v_2$}] (v28) at (27,12);
\coordinate [label=left:{$v_6$}] (v68) at (27,13.5);
\coordinate [label=left:{$v_5$}] (v58) at (27,15);

\coordinate [label=left:{$G_4$}] (G8) at (24.5,8.5);

\end{tikzpicture}
\end{center}
\caption{$B=\{1,5\}$ is a metric basis of $G$, $f\in {\cal S}(B)$ and  $\{G_1,...,G_8\}\subset {\cal{G}}_f$}
\end{figure}
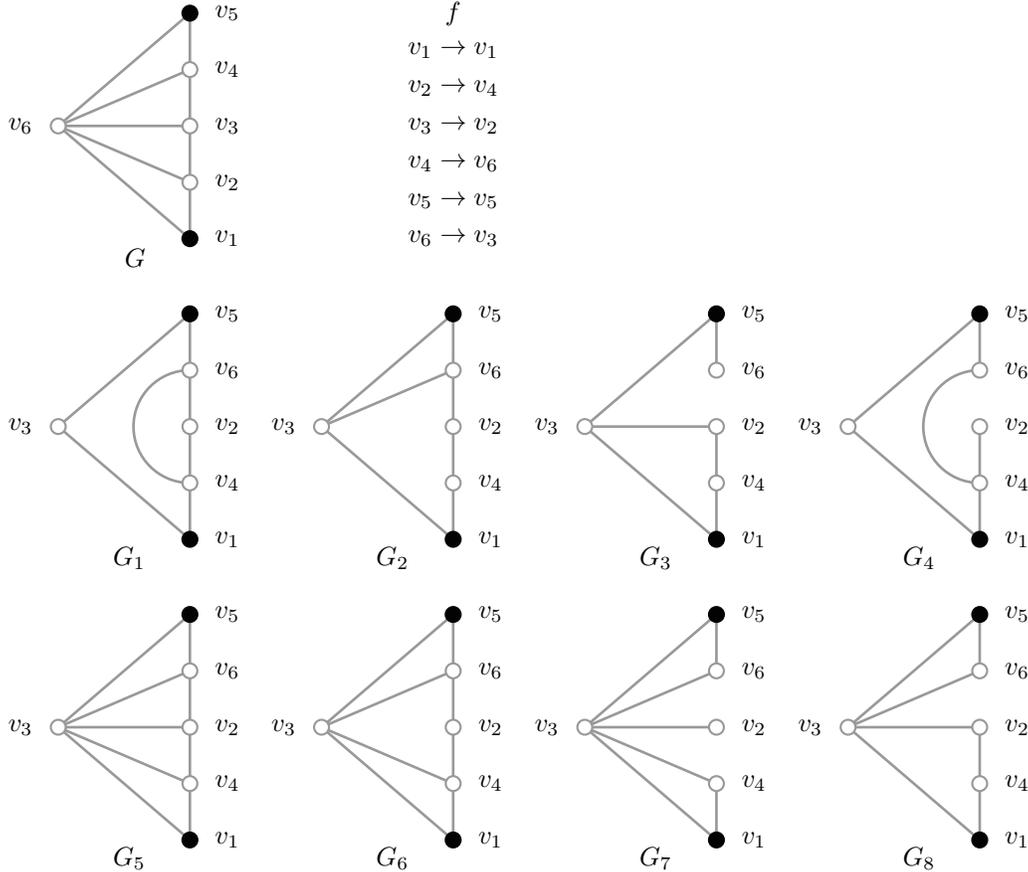

Now, if ${\rm {\bf  B}}_{D(G)-2}(B)=V$, then ${\cal{G}}_B$ consists of graphs isomorphic to each other, having the basis $B$ in common and, as a consequence,  for  any non-empty subfamily ${\cal{H}} \subseteq {\cal{G}}_B$ we have $\Sd({\cal{H}})=\dim(G)$. As the next result shows, this conclusion on $\Sd({\cal{H}})$ need not be restricted to the case ${\rm {\bf{B}}}_{D(G)-2}(B)=V$.


\begin{theorem}\label{Permutation-adimension}
Any  metric  basis $B$ of a connected graph $G$ is  a simultaneous metric generator for  any  family of connected graphs ${\cal {H}} \subseteq {\cal{G}}_B$. Moreover, if   $G \in {\cal{H}}$, then
 $$\Sd({\cal  H})=\dim(G).$$
\end{theorem}

\begin{proof}
Assume that $B$ is a metric basis of a connected graph $G=(V,E)$, $f\in {\cal S}(B)$ and  $G'\in {\cal G}_f$.
We shall show that $B$ is a metric generator for $G'$.
To this end, we take two different vertices $u',v'\in V-B$ of $G'$ and the corresponding vertices $u,v\in V$ of $G$ such that $f(u)=u'$ and $f(v)=v'$.
Since $u\ne v$ and $u,v\not\in B$, there exists  $b\in B$ such that $d_G(u,b)\neq d_G(v,b)$. Now,  consider the following two cases for $u,v$.
\\
\\
\noindent (1) $u,v\in {\rm {\bf  B}}_{D(G)-1}(B)$. In this case, since $d_G(u,b)\neq d_G(v,b)$, Corollary \ref{Lemma-Paths-in-G'} leads to    $d_{G'}(u',b)\neq d_{G'}(v',b)$.
\\
\\
\noindent (2) $u\in {\rm {\bf  B}}_{D(G)-1}(B)$ and $v\not\in {\rm {\bf  B}}_{D(G)-1}(B)$. By Corollary \ref{Lemma-Paths-in-G'}, $d_{G'}(u',b)\le D(G)-1$ and,  if $d_{G'}(v',b)\le D(G)-1$, then $d_{G}(v,b)\le D(G)-1$, which is  not possible since $v\not\in {\rm {\bf  B}}_{D(G)-1}(B)$. Hence, $d_{G'}(v',b)\ge D(G)$ and so $d_{G'}(u',b)\neq d_{G'}(v',b)$.
\\
\\
Notice that since $B$ is a metric basis of $G$,  the case $u,v\not\in {\rm {\bf  B}}_{D(G)-1}(B)$ is not possible.

According to the two cases above, $B$ is a metric generator for $G'$
 and, as a consequence, $B$ is also a simultaneous metric generator for  any  family of connected graphs ${\cal  H} \subseteq {\cal  G}_B$. Thus $\Sd({\cal  H})\le |B| =\dim(G)$ and, if  $G\in {\cal  H}$, then $\Sd({\cal  H})\ge   \dim(G)$. Therefore, the result follows.
\end{proof}

Figure \ref{ExSimultaneousBasis}  shows a graph $G$ for which $B=\{v_1,v_5\}$ is a metric basis. The map $f$ belongs to the stabilizer of $B$ and $\{G_1,...,G_8\}$ is a subfamily of ${\cal  G}_f$. In this case, the family  ${\cal  G}_B$ contains $1344$ different connected graphs; $48$ of them are paths and $B$ is a metric basis of the remaining $1296$ connected graphs.

In Theorem \ref{familyCycles} we showed that if a family $\cal{F}$ of cycles of order $n$ on the same vertex set is selected, then the simultaneous metric dimension of this family is guaranteed to be $2$ if $|{\cal{F}}| < n-1$. Moreover, we showed that there is a family of $n-1$ cycles of order $n$ on the same vertex set, whose simultaneous dimension exceeds $2$. In this section we showed that if $G=(V,E)$ is a fixed connected graph with a given basis $B$, then there is a large number of distinct labeled graphs $G'$ such that $G$ and $G'$ share a large common induced subgraph and such that the simultaneous dimension of this family is $\dim(G)$.  
Thus if $G$ is a cycle of even order $n \ge 6$ and the vertices in the metric basis $B$ are adjacent, then ${\cal  G}_B$ consist of $2(n-2)!$  connected labelled graphs; half of them are   cycles and the remaining are  paths of order $n$. If the vertices in the metric basis $B$ are not adjacent, then ${\cal  G}_B$ consist of $(n-2)!$ cycles of order $n$.

\section{The Simultaneous Metric Dimension of Trees}

It is known, see \cite{Chartrand2000}, that the metric dimension of any given tree can be computed in polynomial time.
To describe one such algorithm, we begin by defining a few terms. A vertex of degree at least $2$ in a graph $G$ is called an {\em interior vertex}.
The set of interior vertices of graph $G$ is denoted by ${\mathcal I}(G)$. A vertex of degree at least 3 is called a \textit{major vertex} of $G$.
Any leaf $u$ of $G$ is said to be a {\em terminal vertex of a major vertex} $v$ of $G$ if $d(u,v)<d(u,w)$ for every other major vertex $w$ of $G$.
The {\em terminal degree} $ter(v)$ of a major vertex $v$ is the number of terminal vertices of $v$, i.e., the number of paths in $G-v$.
A major vertex $v$ of $G$ is an \textit{exterior major vertex} of $G$ if it has positive terminal degree. The set of exterior major vertices of graph $G$ is denoted by ${\mathcal M}(G)$.
It was shown in \cite{Chartrand2000} that a metric generator $W$ of a tree $T$ may be constructed as follows: for each exterior major vertex of $T$
select a vertex from each of the paths of $T-v$ except from exactly one such path and place it in $W$. So $\dim(T) = \sum_{w \in {\mathcal M}(T)} (ter(w)-1)$.

It is natural then to ask whether the simultaneous metric dimension of families of trees can be found in polynomial time.
In this section we show that this is a difficult problem. We obtain sharp bounds for the metric dimension of any given collection of trees and for families of so called `dynamic tree networks'.

\subsection{Computability of the Simultaneous Metric Dimension for Trees}

We show that the problem of finding the simultaneous metric dimension (when stated as a decision problem) is $NP$-complete for families of trees.

\medskip

\noindent{\bf Simultaneous Metric Dimension (SMD)}\\
INSTANCE: A family ${\mathcal{G}}=\{G_1,G_2, \ldots, G_k\}$ of (labeled) graphs on the same vertex set $V$ and integer $B$, $1 \le B \le |V|-1$.\\
QUESTION: Is $\Sd({\mathcal{G}}) \le B$?

\medskip

We use the transformation from the {\bf Hitting set Problem} which was shown to be NP-complete by Karp \cite{Karp1972}.

\medskip
 \noindent{\bf Hitting Set Problem (HSP)}\\
INSTANCE: A collection $C$ of nonempty subsets of a finite set $S$ and a positive integer $k \le |S|$.\\
QUESTION: Is there a subset $S' \subseteq S$ with $|S'| \le K$ such that $S'$ contains at least one element from each subset in $C$?

\medskip

\begin{theorem}\label{NP_completeness_of_SMD}
The Simultaneous Metric Dimension Problem (SMD) is NP-complete for families of trees.
\end{theorem}
\begin{proof}
It is easily seen that SMD is in $NP$.

Let $C=\{C_1,C_2, \ldots, C_k\}$ be a family of nonempty subsets of a finite set $S=\{v_1,v_2, \ldots, v_n\}$ and let $K$ be a positive integer such that $K \le |S|$.
Let $U= \{u_1,u_2\}$, $W=\{w_1,w_2, \ldots, w_k\}$ and set $V = S \cup U \cup W$ where the sets $S$, $U$ and $W$ are pairwise disjoint.

We now construct a family of $k$ trees $T_1, T_2, \ldots, T_k$ on $V$ as follows:  For each $i$, $1 \le i \le k$ let $P_i$ be a path on the vertices of $C_i$
and $Q_i$ a path on the vertices of $(S-C_i) \cup(W-\{w_i\})$. Let $T_i$ be obtained from $P_i$, $Q_i$ and the vertices $u_1,u_2, w_i$ by joining both $u_1$ and $u_2$
to the one leaf of $Q_i$ and then joining $w_i$ and one leaf of $P_i$ to the other end vertex of $Q_i$, see Figure \ref{labeled trees} for an illustration.
Let ${\mathcal{T}} =\{T_1, T_2, \ldots, T_k\}$ and let $B = K+1$. Then there is a subset $S'$ of $S$ with $|S'| \le K$ such that $S'$ contains at least one element
from each $C_i$, $1 \le i \le k$ if and only if $\Sd({\mathcal{T}}) \le B$. This way we have described a polynomial transformation of HSP to SMD.
\end{proof}

\begin{figure}[h]
\begin{center}

\begin{tikzpicture}
[inner sep=0.7mm, place/.style={circle,draw=black!40,
fill=white,thick},xx/.style={circle,draw=black!99,fill=black!99,thick},
transition/.style={rectangle,draw=black!50,fill=black!20,thick},line width=1pt,scale=0.5]
\coordinate (A1) at (0,0);
\coordinate (B1) at (0,2);
\coordinate (C1) at (1,1);
\coordinate (D1) at (2,1);
\coordinate (E1) at (3,1);
\coordinate (F1) at (4,1);
\coordinate (G1) at (5,0);
\coordinate (H1) at (5,2);
\coordinate (I1) at (6,3);
\coordinate (J1) at (7,4);

\coordinate (A2) at (9,0);
\coordinate (B2) at (9,2);
\coordinate (C2) at (10,1);
\coordinate (D2) at (11,1);
\coordinate (E2) at (12,1);
\coordinate (F2) at (13,1);
\coordinate (G2) at (14,0);
\coordinate (H2) at (14,2);
\coordinate (I2) at (15,3);
\coordinate (J2) at (16,4);

\coordinate (A3) at (18,0);
\coordinate (B3) at (18,2);
\coordinate (C3) at (19,1);
\coordinate (D3) at (20,1);
\coordinate (E3) at (21,1);
\coordinate (F3) at (22,1);
\coordinate (G3) at (23,1);
\coordinate (H3) at (24,0);
\coordinate (I3) at (24,2);
\coordinate (J3) at (25,3);

\draw[black!40] (A1) -- (C1) -- (D1) -- (E1) -- (F1) -- (H1) -- (I1) -- (J1);
\draw[black!40] (B1) -- (C1);
\draw[black!40] (G1) -- (F1);

\draw[black!40] (A2) -- (C2) -- (D2) -- (E2) -- (F2) -- (H2) -- (I2) -- (J2);
\draw[black!40] (B2) -- (C2);
\draw[black!40] (G2) -- (F2);

\draw[black!40] (A3) -- (C3) -- (D3) -- (E3) -- (F3) -- (G3) -- (I3) -- (J3);
\draw[black!40] (B3) -- (C3);
\draw[black!40] (G3) -- (H3);

\node at (A1) [place]  {};
\coordinate [label=center:{$u_1$}] (u11) at (-0.5,2.5);
\node at (B1) [place]  {};
\coordinate [label=center:{$u_2$}] (u12) at (-0.5,0.5);
\node at (C1) [place]  {};
\coordinate [label=center:{$v_4$}] (v14) at (1.3,1.5);
\node at (D1) [place]  {};
\coordinate [label=center:{$v_5$}] (v15) at (2.3,1.5);
\node at (E1) [place]  {};
\coordinate [label=center:{$w_2$}] (w12) at (3.3,1.5);
\node at (F1) [place]  {};
\coordinate [label=center:{$w_3$}] (w13) at (3.7,0.5);
\node at (G1) [place]  {};
\coordinate [label=center:{$w_1$}] (w11) at (5.5,0.5);
\node at (H1) [place]  {};
\coordinate [label=center:{$v_3$}] (v13) at (5.5,1.5);
\node at (I1) [place]  {};
\coordinate [label=center:{$v_2$}] (v12) at (6.5,2.5);
\node at (J1) [place]  {};
\coordinate [label=center:{$v_1$}] (v11) at (7.5,3.5);

\node at (A2) [place]  {};
\coordinate [label=center:{$u_2$}] (u22) at (8.5,0.5);
\node at (B2) [place]  {};
\coordinate [label=center:{$u_1$}] (u21) at (8.5,2.5);
\node at (C2) [place]  {};
\coordinate [label=center:{$v_1$}] (v21) at (10.3,1.5);
\node at (D2) [place]  {};
\coordinate [label=center:{$v_5$}] (v25) at (11.3,1.5);
\node at (E2) [place]  {};
\coordinate [label=center:{$w_1$}] (w21) at (12.3,1.5);
\node at (F2) [place]  {};
\coordinate [label=center:{$w_3$}] (w23) at (12.7,0.5);
\node at (G2) [place]  {};
\coordinate [label=center:{$w_2$}] (w22) at (14.5,0.5);
\node at (H2) [place]  {};
\coordinate [label=center:{$v_4$}] (v24) at (14.5,1.5);
\node at (I2) [place]  {};
\coordinate [label=center:{$v_3$}] (v23) at (15.5,2.5);
\node at (J2) [place]  {};
\coordinate [label=center:{$v_2$}] (v22) at (16.5,3.5);

\node at (A3) [place]  {};
\coordinate [label=center:{$u_2$}] (u32) at (17.5,0.5);
\node at (B3) [place]  {};
\coordinate [label=center:{$u_1$}] (u31) at (17.5,2.5);
\node at (C3) [place]  {};
\coordinate [label=center:{$v_1$}] (v31) at (19.3,1.5);
\node at (D3) [place]  {};
\coordinate [label=center:{$v_2$}] (v32) at (20.3,1.5);
\node at (E3) [place]  {};
\coordinate [label=center:{$v_3$}] (v33) at (21.3,1.5);
\node at (F3) [place]  {};
\coordinate [label=center:{$w_1$}] (w31) at (22.3,1.5);
\node at (G3) [place]  {};
\coordinate [label=center:{$w_2$}] (w32) at (22.7,0.5);
\node at (H3) [place]  {};
\coordinate [label=center:{$w_3$}] (w33) at (24.5,0.5);
\node at (I3) [place]  {};
\coordinate [label=center:{$v_5$}] (v35) at (24.5,1.5);
\node at (J3) [place]  {};
\coordinate [label=center:{$v_4$}] (v34) at (25.5,2.5);

\coordinate [label=center:{$T_1$}] (T1) at (3,-1);
\coordinate [label=center:{$T_2$}] (T2) at (12,-1);
\coordinate [label=center:{$T_3$}] (T3) at (21,-1);

\end{tikzpicture}

\caption{A transformation from HSP to SMD for $S=\{v_1,v_2, \ldots,v_5\}$ and $C =\{\{v_1,v_2,v_3\},\{v_2,v_3,v_4\},\{v_4,v_5\}\}$}
\label{labeled trees}
\end{center}
\end{figure}
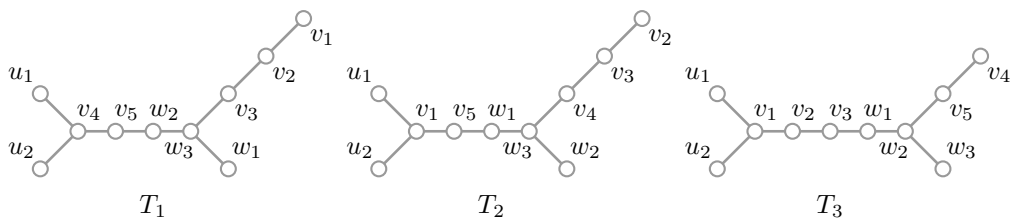

\subsection{Bounds for the Simultaneous metric Dimension of Families of Trees}

The result from the previous section suggest finding bounds for the simultaneous metric dimension of families of trees. We establish next a sharp upper bound for such families.


\begin{proposition}\label{FamilyTrees}
Let ${\mathcal T}=\{T_1, T_2, \ldots , T_k\}$ be a family of trees, which are different from paths,  defined on a common vertex set $V$, and let $S_{{\mathcal I}} = \underset{i=1}{\overset{k}{\bigcap}} {\mathcal I}(T_i)$ be the
set of vertices that are simultaneously interior vertices of every tree $T_i \in {\mathcal T}$. Then $$ \Sd({\mathcal T}) \leq \vert V \vert - \vert S_{{\mathcal I}} \vert - 1. $$
\end{proposition}

\begin{proof}
Using the ideas that underly the validity of the algorithm for constructing a (minimum) resolving set of a tree described in \cite{Chartrand2000}, it is possible to construct a set $S$, which is simultaneously a
metric generator for every tree $T_i \in {\mathcal T}$ by constructing metric generators $W_i$ for every tree $T_i$ as described and letting $S = \underset{i=1}{\overset{k}{\bigcup}} {W_i}$.
Any such set $S$ will not contain a vertex that is not in $S_{{\mathcal I}}$, so $$\Sd({\mathcal T}) \leq \vert S \vert \leq \vert V \vert - \vert S_{{\mathcal I}} \vert$$

Moreover, for every vertex $u \in V - S_{{\mathcal I}}$ and every tree $T_i \in {\mathcal T}$, either:
\begin{enumerate}
\item[(i)]$u$ is a terminal vertex of an exterior major vertex $x$ of $T_i$, in which case every other terminal vertex of $x$, other than $u$, may be selected when constructing $W_i$, and hence $W_i$ may be constructed in such a way that $u \notin W_i$; or
\item[(ii)] $u$ is not an end-vertex of $T_i$, in which case $W_i$ may be constructed in such a way that $u \notin W_i$.
\end{enumerate}

Thus, for every vertex $u \in V - S_{{\mathcal I}}$, the set $S$ may be constructed in such a way that $u \notin S$ and, as a result,  $\Sd({\mathcal T}) \leq \vert S \vert \leq \vert V \vert - \vert S_{{\mathcal I}} \vert - 1$.
\end{proof}

The above inequality is sharp. For instance, equality is achieved for the graph family shown in Figure~\ref{FigTightBoundTreeFamily}, where $S_{\mathcal I} = \{m_1, m_2, i_1\}$, any triple of leaves is a simultaneous metric generator, e.g. $\{e_1, e_2, e_3\}$, whereas no pair of vertices is a simultaneous metric generator. Thus  $\Sd({\mathcal T}) = 3 = \vert V \vert - \vert S_{\mathcal I} \vert - 1$.

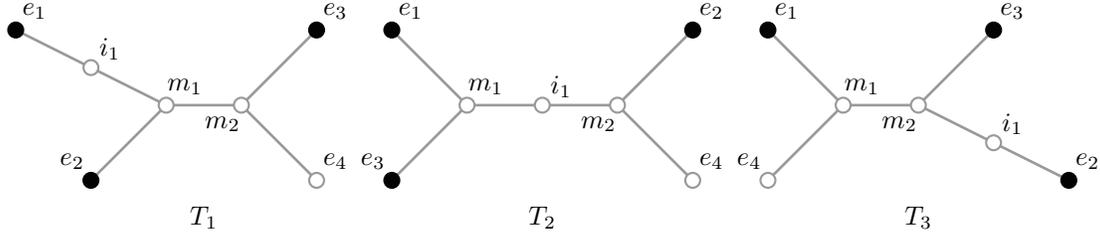
\begin{figure}[h]
\label{FigTightBoundTreeFamily}
\begin{center}
\begin{tikzpicture}
[inner sep=0.7mm, place/.style={circle,draw=black!40,
fill=white,thick},xx/.style={circle,draw=black!99,fill=black!99,thick},
transition/.style={rectangle,draw=black!50,fill=black!20,thick},line width=1pt,scale=0.5]
\coordinate (A) at (-15,6);
\coordinate (B) at (-11,4);
\coordinate (C) at (-13,2);
\coordinate (D) at (-9,4);
\coordinate (E) at (-7,6);
\coordinate (F) at (-7,2);
\coordinate (S) at (-13,5);

\coordinate (G) at (-5,6);
\coordinate (H) at (-3,4);
\coordinate (I) at (-5,2);
\coordinate (J) at (1,4);
\coordinate (K) at (3,6);
\coordinate (L) at (3,2);
\coordinate (U) at (-1,4);

\coordinate (M) at (5,6);
\coordinate (N) at (7,4);
\coordinate (O) at (5,2);
\coordinate (P) at (9,4);
\coordinate (Q) at (11,6);
\coordinate (R) at (13,2);
\coordinate (T) at (11,3);

\draw[black!40] (A) -- (S) -- (B) -- (D) -- (E);
\draw[black!40] (B) -- (C);
\draw[black!40] (D) -- (F);

\draw[black!40] (G) -- (H) -- (U) -- (J) -- (K);
\draw[black!40] (I) -- (H);
\draw[black!40] (J) -- (L);

\draw[black!40] (M) -- (N) -- (P) -- (Q);
\draw[black!40] (N) -- (O);
\draw[black!40] (P) -- (T) -- (R);

\node at (A) [xx]  {};
\coordinate [label=center:{$e_1$}] (e11) at (-14.5,6.5);
\node at (B) [place]  {};
\coordinate [label=center:{$m_1$}] (m11) at (-10.5,4.5);
\node at (C) [xx]  {};
\coordinate [label=center:{$e_2$}] (e12) at (-13.5,2.5);
\node at (D) [place]  {};
\coordinate [label=center:{$m_2$}] (m12) at (-9.5,3.5);
\node at (E) [xx]  {};
\coordinate [label=center:{$e_3$}] (e13) at (-6.5,6.5);
\node at (F) [place]  {};
\coordinate [label=center:{$e_4$}] (e14) at (-6.5,2.5);
\node at (S) [place]  {};
\coordinate [label=center:{$i_1$}] (i1) at (-12.5,5.5);

\node at (G) [xx]  {};
\coordinate [label=center:{$e_1$}] (e21) at (-4.5,6.5);
\node at (H) [place]  {};
\coordinate [label=center:{$m_1$}] (m21) at (-2.5,4.5);
\node at (I) [xx]  {};
\coordinate [label=center:{$e_3$}] (e23) at (-5.5,2.5);
\node at (J) [place]  {};
\coordinate [label=center:{$m_2$}] (m22) at (0.5,3.5);
\node at (K) [xx]  {};
\coordinate [label=center:{$e_2$}] (e22) at (3.5,6.5);
\node at (L) [place]  {};
\coordinate [label=center:{$e_4$}] (e24) at (3.5,2.5);
\node at (U) [place]  {};
\coordinate [label=center:{$i_1$}] (i2) at (-0.5,4.5);

\node at (M) [xx]  {};
\coordinate [label=center:{$e_1$}] (e31) at (5.5,6.5);
\node at (N) [place]  {};
\coordinate [label=center:{$m_1$}] (m31) at (7.5,4.5);
\node at (O) [place]  {};
\coordinate [label=center:{$e_4$}] (e34) at (4.5,2.5);
\node at (P) [place]  {};
\coordinate [label=center:{$m_2$}] (m32) at (8.5,3.5);
\node at (Q) [xx]  {};
\coordinate [label=center:{$e_3$}] (e33) at (11.5,6.5);
\node at (R) [xx]  {};
\coordinate [label=center:{$e_2$}] (e32) at (13.5,2.5);
\node at (T) [place]  {};
\coordinate [label=center:{$i_1$}] (i3) at (11.5,3.5);

\coordinate [label=center:{$T_1$}] (T1) at (-10,1);
\coordinate [label=center:{$T_2$}] (T2) at (-1,1);
\coordinate [label=center:{$T_3$}] (T3) at (9,1);

\end{tikzpicture}

\end{center}
\caption{A family of trees ${\cal T} = \{T_1, T_2, T_3\}$ such that $\Sd({\cal T}) = 3 = \vert V \vert - \vert S_{\cal I} \vert - 1$.}

\end{figure}

However, there are families $\mathcal{T}$ of trees on the same vertex set for which the ratio $\frac{\Sd({\mathcal{T}})}{\vert V \vert - \vert S_{{\mathcal I}} \vert - 1}$ can be made arbitrarily small.  To see this let $r,s \ge 3$ be integers and let $V=\{(i,j) \vert 1 \le i \le r,~1 \le j \le s\} \cup\{x\}$. So $\vert V \vert = rs+1$. Let $T_1$ be the tree obtained from the paths $Q_i =(i,1)(i,2) \ldots (i,s)x$ for $1 \le i \le r$ by identifying the vertex $x$ from each of the paths. So $T_1$ is isomorphic to the tree obtained from the star $K_{1,r}$ by subdividing each edge $s-1$ times. For $2 \le j < s$ let $T_j$ be obtained from $T_1$ by adding the edge $(i,1)(i,j+1)$ and deleting the edge $(i,j)(i,j+1)$ for $1 \le i \le r$. Finally let $T_s$ be obtained from $T_1$ by adding the edge $(i,1)x$ and deleting the edge $(i,s)x$ for $1 \le i \le r$. Let ${\mathcal{T}}=\{T_j \vert 1\le j \le s\}$. Then $S_{\mathcal I} =\{x\}$. So $\vert V \vert -\vert S_{\mathcal I} \vert - 1 = rs-1$. It is not difficult to see that $\{(i,1) \vert 1 \le i \le r-1\}$ is a minimum resolving set for each $T_j$. Hence $\Sd({\mathcal{T}}) = r-1$. So $\frac{\Sd({\mathcal{T}})}{\vert V \vert - \vert S_{{\mathcal I}} \vert - 1} = \frac{r-1}{rs-1}$. By choosing $s$ large enough this can be made as small as we wish. Note also that this family of trees achieves the lower bound given in Observation \ref{cotaTrivialSimultaneous}.

\subsection{Simultaneous Resolving Sets in Dynamic Tree Networks}

Motivated by the results of the previous section we obtain sharp upper and lower bounds for the simultaneous metric dimension of families of trees on the same vertex set that can be obtained by starting from a given tree and making repeated small changes.  We say that a tree $T_2$ is obtained from a tree $T_1$ by an {\em  edge exchange} if there is an edge $e_1$ not in $T_1$ and an edge $f_1$  in $T_1$ such that $T_2 = T_1 +e_1-f_1$.

\begin{theorem}\label{tree_edge_exchange}
Let $T_1$ and $T_2$ be trees where $T_2$ is obtained from $T_1$ by an edge exchange. Then \[\dim(T_1) -2 \le \dim(T_2) \le \dim(T_1)+2.\]
If the upper bound is attained, then there is a metric basis for $T_2$ that contains a metric basis for $T_1$ and if the lower bound is attained, then there is a metric basis for $T_1$ that contains a metric basis for $T_2$. Moreover, these bounds are sharp.
\end{theorem}
\begin{proof}
Suppose $T_2 = T_1 +e_1-f_1$ where $e_1$ is an edge not in $T_1$ and $f_1$ is an edge in $T_1$. Let $e_1=uv$ and $f_1=xy$. By an earlier comment a metric basis for $T_2$ can be constructed by selecting for each major vertex, with positive terminal degree, all but one of its terminal vertices. Since $T_2$ is obtained from $T_1$ by an edge exchange it follows that the metric dimension of $T_2$ can be at most $2$ more than the metric dimension of $T_1$ since $T_2$ has at most two more leaves than $T_1$ and this upper bound is achieved only if $\sum_{w \in {\mathcal M}(T_2)} (ter(w)-1)=\sum_{w \in {\mathcal M}(T_1)} (ter(w)-1)+2$. Similarly, the metric dimension of $T_2$ can be at most two less than the metric dimension of $T_1$ and this lower bounds is achieved only if $u$ and $v$ are both terminal vertices of distinct exterior major vertices and $\sum_{w \in {\mathcal M}(T_2)} (ter(w)-1)+2=\sum_{w \in {\mathcal M}(T_1)} (ter(w)-1)$.

If the upper bound is attained, then a metric basis of $T_1$ together with $x$ and $y$ is a metric basis for $T_2$ with the specified properties. Also, if the lower bound is attained, then a metric basis for $T_2$ together with $u$ and $v$ is a metric basis for $T_1$ with the desired properties.

For the sharpness of the upper bound let $T_1'$ be the tree obtained from the star $K_{1,3}$ having center $v_1$ and leaves $u, v_{1,1}, v_{1,2}$ by subdividing the edge $uv_1$ twice. Let $x_1$ be the neighbour of $u$ in $T_1'$ and $y_1$ the neighbour of $v_1$ in $T_1'$. Now let $T$ be any tree that contains a vertex labeled $u$ but which is otherwise vertex disjoint from $T_1'$. Identify the vertex labeled $u$ in $T_1'$ with the vertex labeled $u$ in $T$ and join a new leaf $y$ to $u$. Let $T_1$ be the resulting tree. Now let $e_1 =uv_1$ and $f_1=x_1y_1$. If $T_2 = T_1 +e_1-f_1$, then $\dim(T_2) = \dim(T_1)+2$.

For the lower bound let $T_1$ be a double star obtained from two copies of the star $K_{1,r}$, $r \ge 3$ with centers $x$ and $y$, respectively by joining $x$ and $y$ with an edge. Let $u$ be a neighbour of $x$ and $v$ a neighbour of $y$. Let $e_1=uv$ and $f_1 =xy$ and let $T_2 = T_1+e_1-f_1$. Then $\dim(T_1) - 2 = \dim(T_2)$.

\end{proof}

\begin{corollary}
Let $T_1,T_2, \ldots, T_k$ be a sequence of trees such that $T_{i+1}$ is obtained from $T_i$ by an edge exchange for $1 \le i < k$. Then \[\dim(T_1) -2k \le \dim(T_k) \le \dim(T_1)+2k.\] If the upper bound is attained, there is a metric basis  for $T_k$ that contains a metric basis for $T_i$ for all $1 \le i \le k$ and if the lower bound is attained, there is a metric basis of $T_1$ that contains a metric basis for $T_i$ for $1 \le i \le k$. Moreover, these bounds are sharp.
\end{corollary}
\begin{proof} The bounds follow from Theorem \ref{tree_edge_exchange} and induction. Also, if the upper bound is attained, there is a metric basis  for $T_k$ that contains a metric basis for $T_i$ for all $1 \le i \le k$ and if the lower bound is attained, there is a metric basis of $T_1$ that contains a metric basis of $T_i$ for $1 \le i \le k$.

We now illustrate the sharpness of the given bounds. Consider first the upper bound.  Let $T_i'$ be the tree obtained from the star $K_{1,3}$, with center $v_i$ and leaves $u, v_{i,1}, v_{i,2}$, by subdividing the edge $uv_i$ twice. Let $x_i$ be the neighbour of $u$ in $T_i'$ and $y_i$ the neighbour of $v_i$ in $T_i'$. Identify the vertices labeled $u$ in each $T_i$ and then join a leaf $y$ to $u$ and let $T_1$ be the resulting tree. For $i = 1,2, \ldots, k$ let $e_i =uv_i$ and $f_i=x_iy_i$. Suppose $T_i$ has been defined for some $1 \le i <k$.  Let $T_{i+1} = T_i+e_i-f_i$. Then $\dim(T_k) \le \dim(T_1)+2k$.

For the lower bound let $T_1$ be obtained from a star $K_{1,2k+1}$ with center $x$ and leaves $\{u_1, u_2, \ldots, u_k,$ $ u_{k+1}\} \cup \{y_1, y_2, \ldots, y_k\}$ by joining three (new) leaves $v_i,z_i$ and $s_i$ to $y_i$ for $1 \le i \le k$. For $1 \le i \le k$ let $e_i = u_iv_i$ and $f_i = xy_i$. Suppose $T_i$ has been defined for some $i$, $1 \le i < k$. Let $T_{i+1} = T_i +e_i-f_i$. Then $\dim(T_1) -2k = \dim(T_k)$.

\end{proof}

\begin{corollary}
If $T_1$ and $T_2$ are any two trees on the same set of $n$ vertices such that $T_1$ and $T_2$ have $k$ edges in common where $k > n/2$, then \[\dim(T_1) -2(n-k-1) \le \dim(T_2) \le \dim(T_1)+2(n-k-1).\]
\end{corollary}
\begin{proof} This follows from the fact that $T_2$ can be obtained from $T_1$ by a sequence of $n-k-1$ edge exchanges.
\end{proof}

\section{Concluding Remarks}

In this paper we obtained sharp upper and lower bounds on the simultaneous metric dimension of families of (connected) graphs. We obtained exact values for this invariant in the case of some specific families of graphs. We showed that if $T_1, T_2,  \ldots, T_k$ is a sequence of trees such that each $T_i$, $2 \le i \le k$ is obtained from $T_{i-1}$ by an edge-exchange, then the metric dimension of the family differs from that of $T_1$ by at most $2k$. We showed, for a given connected graph $G$, that there is a large family of labeled graphs with the same vertex set having the same simultaneous metric dimension as $G$. Nevertheless it appears to be a difficult problem in general to find the exact value for the simultaneous metric dimension of a family of graphs on the same vertex set, even if the metric dimension of each member of the family is known.

\bibliographystyle{elsart-num-sort}

\end{document}